\documentclass[11pt]{amsart}
\usepackage{amscd,amssymb,amsmath,enumerate}
\usepackage{amsmath,amsthm}
\usepackage[all]{xy}
\usepackage{graphicx}
\usepackage{graphics}
\usepackage{latexsym}
\usepackage[all]{xy}
\usepackage{psfrag}
\xyoption{matrix} \xyoption{arrow}
\usepackage{parskip}
\usepackage{tikz-cd}
\usepackage{amsfonts}

\newtheorem{proposition}{Proposition}[section]
\newtheorem{theorem}[proposition]{Theorem}
\newtheorem{lemma}[proposition]{Lemma}
\newtheorem{corollary}[proposition]{Corollary}

\newtheorem{definition}[proposition]{{Definition}}
\newenvironment{defn}{\begin{definition} \rm}{\end{definition}}

\newtheorem{remark}[proposition]{{Remark}}

\newtheorem{Question}[proposition]{Question}

\newtheorem{Examples}[proposition]{Examples}

\newcommand{\cA}{{\mathcal A}}
\newcommand{\cB}{{\mathcal B}}

\newcommand{\cX}{{\mathcal X}}

\newcommand{\cP}{{\mathcal P}}
\newcommand{\cQ}{{Q}}

\newcommand{\cM}{{\mathcal M}}

\newcommand{\cH}{{\mathcal H}}
\newcommand{\cT}{{\mathcal T}}

\newcommand{\cE}{{\mathcal E}}
\newcommand{\cC}{{\mathcal C}}

\newcommand{\cS}{{\mathcal S}}

\renewcommand{\aa}

\newcommand{\ed}{\textcolor{blue}}

\newcommand{\Hom}{\operatorname{Hom}\nolimits}

\newcommand{\soc}{\operatorname{soc}\nolimits}
\newcommand{\rad}{\operatorname{rad}\nolimits}

\renewcommand{\dim}{\operatorname{dim}\nolimits}

\newcommand{\bc}{Brauer configuration}
\newcommand{\bca}{\bc\ algebra}

\usepackage{color}
\definecolor{candyapplered}{rgb}{1.0, 0.03, 0.0}
\makeatletter
\def\thm@space@setup{%
  \thm@preskip=0.7cm \thm@postskip=0.3cm
}
\makeatother

\begin{document}

\nocite{*}

\date{today}
\title[Almost gentle algebras and trivial extensions]
{Almost gentle algebras and their trivial extensions}

\author[Green]{Edward L.\ Green}
\address{Edward L.\ Green, Department of
Mathematics\\ Virginia Tech\\ Blacksburg, VA 24061\\
USA}
\email{green@math.vt.edu}
\author[Schroll]{Sibylle Schroll}
\address{Sibylle Schroll\\
Department of Mathematics \\
University of Leicester \\
University Road \\
Leicester LE1 7RH \\
United Kingdom}
\email{schroll@le.ac.uk }

\thanks{This work was supported through the Engineering and Physical Sciences Research Council, grant numbers EP/K026364/1 and EP/P016294/1}

\begin{abstract}
In this paper we define almost gentle algebras. They are monomial  special multiserial algebras generalizing gentle algebras. We show that the trivial extension of an almost gentle algebra by its minimal injective co-generator is a symmetric special multiserial algebra and hence a Brauer configuration algebra. Conversely, we show that  any almost gentle algebra is an admissible cut of a unique Brauer configuration algebra and as a consequence, we obtain that every Brauer configuration algebra with multiplicity function identically one, is the trivial extension of an almost gentle algebra. We show that to every almost gentle algebra $A$ is associated a hypergraph, and that this hypergraph induces  the Brauer configuration of the trivial extension of $A$.  Amongst other things, this gives a combinatorial criterion to decide when two almost gentle algebras have isomorphic trivial extensions.
\end{abstract}
\date{\today}
\maketitle

\setcounter{tocdepth}{1}

\section{Introduction}

In this paper we introduce a new class of multiserial algebras called almost gentle algebras. These algebras are monomial quadratic algebras which generalise gentle algebras. Namely, an algebra $KQ/I$ is almost gentle if it is special multiserial and if $I$ is generated by paths of length 2. It is clear from the definition that every gentle algebra is almost gentle. While gentle algebras are of tame representation type, almost gentle algebras are wild in general. However, there are many examples of almost gentle algebras of finite and tame representation type  that are not gentle. 

Gentle algebras are one of the classes of most studied algebras 
as they appear in many different contexts such as Jacobian algebras of unpunctured surfaces in cluster theory \cite{ABCP, Labardini}, algebras associated to dimer models \cite{Bocklandt1, Bocklandt2} or in the context of the study of the enveloping algebra of Lie algebra \cite{Khovanov}. Their representation theory comes with a strong combinatorial structure. They are string algebras and as such their indecomposable modules are given by string and band modules, and their Auslander-Reiten quiver is completly determined by the string combinatorics \cite{BR}. Maps between string and band modules have been given in \cite{CB} and \cite{Kr}, respectively. It is a class of algebras closed under derived equivalence \cite{SZ} and they are derived tame. The   indecomposable objects in the derived category of a gentle algebra have been determined in \cite{BM}, they are given by homotopy strings and bands. In \cite{ALP} the maps between homotopy strings and bands have been explicitly described. The singularity category of a gentle algebra has been described in \cite{Ka}. 
 Recently, in \cite{CS, CPS} a basis of the extensions between string and band modules has been given. 

Almost gentle algebras do not have the underlying string combinatorics that gentle algebras have. However, the strong similarity  in their structure makes this an interesting new class of algebras to consider. It contains many examples of well-studied algebras, such as hereditary algebras arising from many orientations of all Dynkin and extended Dynkin quivers. 

Just as gentle algebras, almost gentle algebras can be of finite or infinite global dimension. They are of infinite global dimension if and only if the quiver contains an oriented cycle in which every subpaths of lengths two is a relation. While gentle algebras are Gorenstein \cite{GR}, this is not necessarily true for almost gentle algebras, see the example in Section 2.

In Sections 2 and 3, we give a closed formula for the dimensions of almost gentle algebras and their trivial extensions in terms of maximal paths in the almost gentle algebras.

 Examples of trivial extensions of almost gentle algebras appear, for example, in the derived equivalence classification of symmetric algebras of finite and tame representation type, see \cite{Skowronski} and the references within. In the case of gentle algebras, there is   a characterisation  through their trivial extensions by their minimal injective co-generator. That is, an algebra $A$ is gentle if and only if the trivial extension $T(A) = A \ltimes \Hom_k(A,k) $ is special biserial \cite{PS}, see also \cite{R}.  In Section 4 we show that the trivial extension of an almost gentle algebra is special multiserial.  
 We note, however, that the converse  remains an open question. That is it is not know that if the trivial extension of an algebra $A$  is special multiserial whether this implies that $A$ is almost gentle.

Another class of examples of algebras arising as trivial extensions of almost gentle algebras is given by symmetric algebras with radical cubed zero which have been extensively studied, see for example \cite{B, ESch, ESo} and \cite{GS2}. It follows from the results in this paper and in \cite{GS2} that an algebra is a  symmetric algebra with radical cubed zero  if and only if it is  a trivial extension of an almost gentle algebra where the paths in the quiver of the almost gentle algebra are all of length
at most one.

In Section 5 we show that an admissible cut, as defined in \cite{S}, and based on the definition of admissible cuts in \cite{F} and \cite{FP} (see Section 5 for the definition), of a symmetric special biserial algebra  gives rise to an almost gentle algebra. In the other direction we show that every symmetric special multiserial algebra with no powers in the relations or equivalently that every Brauer configuration algebra with multiplicity function equal to one, is the trivial extension of an almost gentle algebra (see \cite{GS1} for the definition of Brauer configuration algebras). We note that this almost gentle algebra is not unique. In fact, our construction gives a whole family of almost gentle algebras that have isomorphic trivial extensions. While all of these gentle algebras have the same number of simple modules, they can have very different homological properties. For example, some might have finite global dimension while others might have infinite global dimension. Furthermore, it is straightforward to see that these algebras are not derived equivalent in general.   
We leave it as an open question to the reader to determine the relationship between all the gentle algebras that have the same trivial extension. 

 In Section 6, we give a construction of the Brauer configuration of the trivial extension of an almost gentle algebra. The construction is based on the notion of an algebra defined by cycles. 

 A Brauer configuration is a vertex decorated hypergraph with an orientation. Based on this observation, we associate, in Section 7, a decorated hypergraph with orientation to every almost gentle algebra and show that this hypergraph is precisely the Brauer configuration of the trivial extension of the almost gentle algebra. That  is, in the terminology of Section 6, it is exactly the Brauer configuration of the algebra defined by cycles isomorphic to the trivial extension of the  almost gentle algebra.  It follows that two almost gentle algebras have the same trivial extensions if and only if they have the same associated hypergraph.




\section{Almost gentle algebras}

In this section we define almost gentle algebras, generalizing the class of gentle algebras. 

First we fix some notation. Let $K$ be a field. All algebras are assumed to be indecomposable $K$-algebras. Unless otherwise stated, an algebra given by quiver and relations, $KQ/I$ is assumed to be finite dimensional and the ideal $I$ is assumed to be admissible. For a quiver $Q$, we denote by $Q_0$ the set of vertices in $Q$ and by $Q_1$, the set of arrows in $Q$. We set $e_v$ to
be the trivial path at a vertex $v \in Q_0$. Furthermore, for $a, b \in Q_1$, we write $ab$ for the path $a$ followed by $b$. We let $s(a)$ be the vertex at which  the arrow $a$ starts and let $t(a)$ be the vertex at which $a$ ends. For a path $p = a_1...a_n$ in $Q$, we set $s(p) = s(a_1)$ and $t(p) = t(a_n)$. Given a finite dimensional algebra $\Lambda$, let $\Lambda^e \simeq \Lambda \otimes_k \Lambda^{op}$. 

 An algebra $A$ is {\it gentle} if it is Morita equivalent to an algebra $KQ/I$ such that 
\begin{itemize}
\item[(S0)] $I$ is generated by paths of lenghts 2,
 \item[(S1)] for every arrow $a \in Q_1$, there exists at most one arrow $b$ such that $ab \notin I$ and at most one arrow $c$ such that $ca \notin I$, 
\item[(S2)]  for every arrow $a \in Q_1$, there exists at most one arrow $b$ such that $ab \in I$ and at most one arrow $c$ such that $ca \in I$,
\item[(S3)] for every vertex $v \in Q_0$ there are at most two arrows ending at $v$ and at most two arrows starting at $v$.
\end{itemize}


Recall from \cite{GS2} that an algebra is \emph{special multiserial} if it is Morita equivalent to an algebra $KQ /I$ satisfying  condition (S1).

\begin{defn}
We say that an algebra is \emph{almost gentle} if it is Morita equivalent to an algebra $KQ/I$ such that 
\begin{itemize}
\item[(S0)] $I$ is generated by paths of lengths 2,
 \item[(S1)] for every arrow $a \in Q_1$, there exists at most one arrow $b$ such that $ab \notin I$ and at most one arrow $c$ such that $ca \notin I$.
\end{itemize}
\end{defn}
 
So an algebra is almost gentle if  
  it is Morita equivalent to a special multiserial algebra $KQ/I$ where $I$ is generated by monomial relations of length two.

\begin{remark} {\rm Every gentle algebra is almost gentle. }\end{remark}

We state some basic facts about almost gentle algebras.
An almost gentle algebra $KQ/I$  is of infinite global dimension if there is an oriented cycle in $Q$ such that every path of length two in that cycle is in $I$. If no such cycle exists then $KQ/I$ is of finite global dimension.  Since the ideal $I$ can be generated by paths of length 2,  every almost
gentle algebra is a Koszul algebra. 
The only almost gentle algebras that are self-injective are $K[x]/(x^2)$ and the oriented cycle with
all paths of length 2 being relations. 

 Gentle algebras are Gorenstein \cite{GR}. The same does not hold for almost gentle algebra. Consider, for example, the algebra with quiver 
$\xymatrix{
\bullet\ar@/^/[r]_a&\bullet\ar@/^/[l]^b\ar[r]^c&\bullet
}$ and where the ideal of relations is generated by all paths of length 2, then the resulting algebra is almost gentle
but not Gorenstein.

In the following lemmas we collect some obvious properties of almost gentle algebras. 

\begin{lemma}\label{oriented cycle path in I} Let $A = KQ/I$ be an almost gentle algebra and let $C$ be an oriented cycle in $Q$. Then there exists a path of length two in $C$ that is in $I$. 
\end{lemma} 

Let $p$ be a path in $Q$.
 Then we say that $p$ is a \emph{maximal path of $A=KQ/I$} if $p \notin I$ and,
for every arrow $a$ in $Q$ we have $ap \in I$ and $pa \in I$. We denote the set
of maximal paths of $A$ by $\cM$.

\begin{lemma}\label{unique maximal path} Let $A = KQ/I$ be an almost gentle algebra and let $v$ be a
 vertex in $Q$. Then $v$ lies in a unique maximal path of A if and only if one of the following conditions holds

(i) $v$ is a sink with a unique arrow ending at $v$, 

(ii) $v$ is a source with a unique arrow starting at $v$, 

(iii) there is a unique arrow $a$ ending at $v$ and there is a unique arrow $b$ starting at $v$ and $ab \notin I$.
\end{lemma}

\begin{lemma}Let $A=KQ/I$ be an almost gentle algebra. Then 
 
 (i) Every arrow $a \in Q_1$ lies in exactly one maximal path of  A. 
 
 (ii) Let $m\in\cM$. Then $m$ has no repeated arrows.
\end{lemma}

We introduce two functions associated to an almost gentle algebra $A$ 
which will
be used later in the paper.  
Let $\diamond$ be some element not in $Q_1$ and 
set $\cA=Q_1 \cup \{\diamond\}$.  Define
$\sigma\colon Q_1\to \cA$ and $\tau\colon Q_1\to \cA$ by
$\sigma(a)=\begin{cases} b& \text{ if } ab\notin I\\
                \diamond &\text{ if }ab\in I \text{ for all }b\in Q_1\end{cases}
$ and
$\tau(a)=\begin{cases} c& \text{ if } ca\notin I\\
                \diamond &\text{ if }ca\notin I \text{ for all }c\in Q_1\end{cases}$
where $a,b,c\in Q_1$. From the definitioin of special multiserial, we see that
these functions are well-defined. Since $A$ is finite dimensional,  
 for every $a \in Q_1$, there are smallest nonnegative integers $M_a$ and $N_a$ such that 
$\sigma^{M_a}(a)=\diamond$ and $\tau^{N_a}(a)=\diamond$.  It follows
that the unique maximal path of $A$ containing the arrow $a$ is
$\tau^{N_a-1}(a)\tau^{N_a-2}(a)\cdots\tau(a)a\sigma(a)\cdots\sigma^{M_a-1}(a)$, which
is of length $M_a+N_a-1$.
Since a  maximal path of $A$ has no repeated arrows and since every arrow is in a unique
maximal path of $A$, it is easy to see that a maximal  path of $A$ is the unique maximal path 
of $A$ of any of its
arrows and the position at which that arrow occurs in the path is uniquely determined. 

If $A=KQ/I$ and $\pi\colon KQ\to A$ is the canonical surjection, then
for $x\in KQ$ we will denote $\pi(x)$ by $\bar x$.  If $a\in Q_1$, we let $U_a$ be the right $A$-module $\overline{a}A$ generated
by $\overline{a}$.   If $A$  is an almost gentle algebra, then the $U_a$ are uniserial
$A$-modules. Note that this holds more generally if $A$ is  a special multiserial algebra, see \cite{GS2}.

\begin{proposition}\label{radical-structure} Let $A=KQ/I$ be an almost gentle algebra. Then
$\rad(A) =\bigoplus_{a\in Q_1} U_a$.  A
$K$-basis for $A$ is the set of $\bar p$, where  $\bar p$  is a subpath of length $\ge 1$ of a maximal path
of $A$,    together with the trivial subpaths $\overline{e_v}$, for $v\in Q_0$.
\end{proposition}

\begin{proof}
Since $A$ is a  monomial algebra, $A$ has a $K$-basis  $\{ \bar p \mid
p \text{ is a path  in $Q$ and }p\notin I\}$.  Any such $p$ either has length $\ge 1$ or $p=e_v$ for some
vertex $v$.  This proves the basis part of the result.  If $p=a_1\cdots a_n$ is a path in $Q$ then
$\bar p\in U_{a_1}$.  It also follows that
$\sum_{a\in Q_1} U_{a}=\rad(A)$ since the image of the arrows generate $\rad(A)$.  Using
that $A$ is a monomial algebra, we see $\bigoplus_{a\in Q_1} U_{a}=\rad(A)$.
\end{proof}

 If $p$  is a  path in $Q$, we let $\ell(p)$ denote the length of $p$.

\begin{corollary}\label{dimension result}Let $A$ be  an almost
  gentle algebra.
Then
\[ \dim_K(A)=\mid Q_0 \mid + \sum_{m \in\cM}
\ell(U_{m})(\ell(U_{m})+1)/2,\]
where $U_m$ is the right uniserial $A$-module generated by the
first arrow in $m$.
\end{corollary}

\section{The symmetric special multiserial algebra associated to an almost gentle algebra}\label{symm spec
assoc to A}
  In \cite{GS3} given a special
multiserial algebra $A$, we constructed a
symmetric special multiserial algebra $A^*$  such that $A$ is a quotient 
of $A^*$. b Recall from \cite{GS2} that the class of symmetric special multiserial algebras and the class of Brauer configuration algebras coincide. We refer the reader to \cite{GS1} for the definition of a Brauer configuration algebra. We slightly modify that construction below
in the case of an almost gentle algebra.   
In this section, $A=KQ/I$ will denote an almost gentle algebra where $I$ is 
an ideal
generated by quadratic elements.  Recall that $\cM$ is the
set of maximal paths of $A$.  

We begin by defining a new quiver $Q^*$.  The vertices of $Q^*$ are the same
as $Q$. For each $m\in\cM$, let $a_m$ denote an arrow (not in $Q_1$) from the end 
vertex of $m$ to the start vertex of $m$.   The arrow set of $Q^*$ is $Q_1\cup \{a_m\mid m\in\cM\}$.
Since $Q$ is a subquiver of $Q^*$, we freely view paths in $Q$ as paths in $Q^*$.
For each $m\in\cM$, we obtain a cycle $C_m=ma_m$ in $Q^*$.   We let $\cS$ denote
the set cycles $C^*$ such that $C^*$ is a cyclic permutation of $C_m$ for some $m\in\cM$.
Let $\mu\colon \cS\to\mathbb Z_{>0}$ be defined by $\mu(C^*)=1$, for all $C^*\in\cS$.

We say a cycle in $Q^*$ is  \emph{simple} if the cycle has no repeated
arrows.
Following \cite{GS3}   we say a pair $(\cT,\nu)$
is a \emph{defining  pair in $Q$} if  $\cT$ is a set of simple cycles in $Q$
and $\nu\colon\cT\to \mathbb Z_{>0}$ which satisfy the following conditions:
\begin{enumerate}
\item[D0] If $C$ is a loop at a vertex $v$ and $C\in\cT$, then $\nu(C)> 1$.
\item[D1] If a simple cycle is in $\cT$, every cyclic permutation of the cycle is in $\cT$.
\item[D2] If $C\in\cT$ and $C'$ is a cyclic  permutation of $C$ then $\nu(C)=\nu(C')$.
\item[D3] Every arrow occurs in some simple cycle in $\cT$.
\item[D4]  If an arrow occurs in two cycles in  $ \cT$, the cycles are cyclic permutations
of each other.
\end{enumerate}

\begin{proposition}\label{defining pair} The pair $(\cS,\mu)$ defined above is
a defining pair.
\end{proposition}

\begin{proof} Since for $m\in\cM$,  $m$ has no repeated arrows, the cycles $ma_m$ and
their cyclic  permutations are simple cycles. 
If $a$ is a loop in $Q$, and  hence in $Q^*$, then since
  $A$ is finite dimensional
and $I$ can be generated by  paths of length 2, we see that $a^2\in I$.  Let $m$ be the
unique maximal path in which $a$ occurs.  Then $a$ occurs in $ma_m$ which is not
a loop.  That is, $\cS$ contains no loops and hence D0 vacuously holds.  By construction
D1 holds.  Since $\mu \equiv 1$, D2 holds.  Since every arrow in $Q$ occurs in some
maximal  path $m$, every arrow in $Q$ occurs in some cycle $C^*\in\cS$.  Each
new arrow $a_m$ occurs in $ma_m\in\cS$ and we see that D3 holds.  Since
an arrow in $Q$ occurs in a unique maximal path in $Q$, D4 holds.

\end{proof}

Following \cite{GS3},   the defining pair  $(\cS,\mu)$  in $Q^*$ gives rise to a  $K$-algebra with 
quiver $Q^*$ and ideal of relations generated by all relations of the following three
types:
\begin{enumerate}
\item[Type 1] $C^{\mu(C)}-{C'}^{\mu(C')}$,  if $C$ and $C'$ are cycles in $\cS$ at some 
vertex $v\in Q_0$.
\item[Type 2]  $Ca$, if $C\in\cS$ and $a$ is the first arrow in $C$.
\item[Type 3]  $ab$, if $a,b\in Q_1$ and $ab$ does not lay on any $C\in\cS$.
\end{enumerate} 

 In \cite{GS3}  an algebra $KQ/I$ given by  a defining pair $(\cT, \nu)$ in $Q$ such that $I$ is generated  by all relations of Types 1, 2, and 3,  is called \emph{the algebra defined by cycles $(\cT,\nu)$}. 

The following results follows from \cite{GS3}. 

\begin{theorem} Let $A$ be an almost gentle algebra and let $A^*$ be the algebra defined by cycles $(\cS, \mu)$ as defined above. Then $A^*$ 
is a symmetric special multiserial algebra and thus it is a Brauer configuration algebra. 
\end{theorem}

We call the  algebra $A^*=KQ^*/I^*$, above the \emph{symmetric special multiserial
algebra associated to $A$}.

The next result determines the dimension of $A^*$.

\begin{proposition}\label{dim of A*} Let $A$ be an almost gentle algebra and let
$A^*$ be the symmetric special multiserial algebra associated to $A$. Then
\[\dim_K(A^*)=2|Q_0| + \sum_{m\in \cM} \ell(m)\cdot (\ell(m) +1).\]  In
particular, $\dim_K(A^*)=2\dim_K(A)$.
\end{proposition}

\begin{proof}
The quiver of $Q^*$ of
$A^*$ has $|Q_0|$ vertices, and so there are $|Q_0|$ paths of length $0$, the
$e_v$, for $v\in Q_0$. 
Since $A^*$ is a symmetric algebra, the socle of $A^*$ has dimension $|Q_0|$.
We now find the dimension of $\rad(A^*)/\soc(A^*)$.
Consider $m\in\cM$.  The cycle $ma_m\in \cS$ has length $\ell(m)+1$.
If $a$ is an arrow in $ma_m$ then $aA^*$ is a uniserial module of length $\ell(m)+1$.
Then $aA^*/(aA^*\cap \soc(A^*)$ has dimension $\ell(m)$ and there are 
$\ell(m)+1$ choices for $a$.  Therefore,  we see that 
$\dim_K(\rad(A^*)/\soc(A^*))=\sum_{m\in\cM}\ell(m)(\ell(m)+1)$.

The last part follows from Corollary \ref{dimension result}.
\end{proof}

\section{Trivial extension of an almost gentle algebra}\label{sec:trivial extension}

Let $A= KQ/I$ be a finite dimensional algebra and let $D(A)=\Hom_{K}(A,K)$ 
be its $K$-linear dual. Recall that the trivial extension $T(A) = A \rtimes D(A)$
is a symmetric algebra defined as the vector space $A \oplus D(A)$ and with multiplication given by $(a,f)(b,g) = (ab, ag + fb)$, 
for any $a, b \in A$ and $f,g \in D(A)$.  Note that $D(A)$ is an $A$-$A$-bimodule
via the following.  If $a,b\in A$ and $f\in D(A)$,  then $afb\colon A\to K$ by
$(afb)(x)=f(bxa)$.
 We keep the convention that if $x\in KQ$,  and $\pi\colon
KQ\to A$ is the canonical surjection, then we denote $\pi(x)$ by $\bar x$.

Let $\cB$ be the set of finite directed paths in $Q$ and suppose that $I$ is generated by paths;
that is, $KQ/I$ is a monomial algebra.  
Consider the set $\bar\cB=\{p\in \cB\mid p\notin I\}$.  The set $\{\bar p\mid p\in\bar\cB\}$ 
is a $K$-basis of $A$.   We abuse notation and view $\bar \cB$ as a $K$-basis of $A$.  Then
the set $\cM$ of maximal paths of $A$  is a subset of $ \bar\cB$ and forms a $K$-basis of $\soc_{A^e}(A)$. 

The dual basis, $\bar\cB^\vee=\{p^\vee\mid p\in\bar\cB\}$ is a $K$-basis of $D(A)$   where, if $p\in\bar\cB$, 
 $p^\vee\in D(A)$ is the element in $D(A)$ defined by
$p^\vee(q)=\delta_{p,q}$ for $q\in\bar\cB$, where $\delta_{p,q}$ is the Kronecker
delta.

\begin{lemma}\label{basic properties of T(A)} Let $A$ be a 
finite dimensional monomial
algebra with $K$-basis $\bar\cB$ as above.  Then,  
for $p,q,r\in\bar \cB$,  the following
hold  in $T(A)$.
\begin{enumerate}
\item $(p,0)(0,r^\vee)=\begin{cases}(0,s^\vee), &\text{if there is some  }
s\in\bar\cB \text{ with }sp=r\\
0, &\text{otherwise}.
\end{cases}$
\item $(0,r^\vee)(q,0)=\begin{cases}(0,s^\vee), &\text{if there is some  }
s\in\bar\cB \text{ with }qs=r\\
0, &\text{otherwise}.
\end{cases}$
\item  $(0,p^\vee)(q,0)(0,r^\vee)=0$.
\item If  $prq\in\bar\cB$ for some $p,q,r\in\bar\cB$ then
$(q,0)(0,(prq)^\vee)(p,0)=(0,r^\vee)$.
\end{enumerate} 
\end{lemma}

\begin{proof}
Parts (1) and (2) are an immediate consequence of the multiplication
in $T(A)$.  Part (3) follows from parts (1) and (2) and that 
$(0,x^\vee)(0,y^\vee)=0$, for all $x,y\in\bar\cB$.
Part (4) follows from parts (1) and (2).

\end{proof}

\begin{proposition}\label{generating T(A)}
Let $A$ be a finite dimensional monomial algebra. 
Then  $T(A)$  is   generated by $\{(a,0)\mid a\in Q_1\}\cup\{(0,m^\vee)\mid m\in\cM\}$.
\end{proposition}

\begin{proof}
 Let $\bar\cB$ be a $K$-basis of $A$ as defined above. Since $\{(p,0)\mid p\in\bar\cB\}\cup\{(0,p^\vee)\mid p\in\bar\cB\}$ is
a $K$-basis  of $T(A)$, we need only show that if $p\in\bar\cB$, then $(p,0)$ and
$(0,p^\vee)$  are in the two sided ideal  in $T(A)$ generated by
$\{(a,0)\mid a\in Q_1\}\cup\{(0,m^\vee)\mid m\in\cM\}$.
Let $p\in\bar\cB$.  Since $p$ is a product of arrows,
$(p,0)$ is a product of elements of the form $(a,0)$ where 
$a$ is an arrow in $Q$.   Now consider $(0,p^\vee)$.   There
are paths $r$ and $s$ such that $rps\in\cM$.   But then
$(0,p^\vee)=(s,0)(0, (rps)^\vee)(r,0)$ by Lemma~\ref{basic properties of T(A)} part (4) and we are done.
\end{proof}

We now prove the main result of this section.

\begin{theorem}\label{A* iso to T(A)} Let $A=KQ/I$ be an almost gentle algebra,
$A^*$ be the symmetric special multiserial algebra associated to $A$,
and $T(A)$ be the trivial extension of $A$ by $D(A)$.
Then $A^*$ is  isomorphic to $T(A)$.
\end{theorem}

\begin{proof}  Let $Q^*$ be the quiver of $A^*$ which is defined in Section \ref{symm spec assoc to A}.
We begin by defining a ring surjection $\varphi$ from $KQ^*$ to $T(A)$.  Since the vertices of $Q^*$ are
the same as the vertices in $Q$ and since $Q$ is a subquiver of $Q_T$, the quiver of $T(A)$,
we send a vertex $v$ in $Q^*$ to $\bar v$, the image of $v$ in $T(A)$ under the canonical
surjection $KQ_T\to T(A)$.  We define $\varphi$ on arrows as follows.
If $a$ is an arrow in    $ Q\subseteq Q^*$ , let $\varphi(a)=(a,0)$.  If $m\in\cM$,
then $\varphi(a_m)=(0,m^\vee)$.  Note that $a_m$ is an arrow from $t (m)$ to
$s(m)$ and that $(0,m^\vee)=(e_{t(m)},0)(0,m^\vee)(e_{s(m)},0)$.  By the universal 
mapping property of a path algebra, we obtain a $K$-algebra homomorphism
$\varphi\colon KQ^*\to T(A)$.  By  Proposition \ref{generating T(A)}, $\varphi$ is
a surjection.

Next we show that $I^*$ (defined in Section \ref{symm spec assoc to A}) is
contained in $\ker(\varphi) $.   For this,
we prove that relations of Types 1, 2, and 3 are in $\ker(\varphi)$.
Recall that $\cS$ is defined to be the set of simple cycles in $Q^*$ that
are cyclic permutations of the cycles $ma_m$, for some $m\in \cM$.
We begin with a Type 1 relation.
Let $C,C'\in\cS$ be  cycles in $\cS$ at a vertex $v\in Q^*$.  We need to
show that $\varphi(C-C')=0$.  Let $m,m'\cM, p,q,p',q'\in\bar\cB$ such
that $pq=m$, $p'q'=a_{m'}$, $C=qa_mp$, and $C'=q'a_{m'}p'$.
Then  $\varphi(C)=(p,0)(0,m^\vee)(q,0)$.  Since $C$  is a cycle at $v$,
by Lemma \ref{basic properties of T(A)},
$(p,0)(0,m^\vee)(q,0)=(0.qm^\vee p)=(0,r^\vee)$ where $prq=m$.  It follows
that $r=e_v$ since $pe_vq=m$.  Thus, we have shown that
$\varphi(C)=(0,e_v^\vee)$.  By a similar argument, $ \varphi(C')=(0,e_v^\vee)$
and we conclude that $\varphi(C-C')=0$.

Next we show that Type 2 relations are sent to $0$ by $\varphi$.
Let $C\in\cS$ is a cycle at $v$, with first arrow $b$. Either $b$ is an arrow in $Q$ or
$b=a_m$ for some $m\in\cM$.  Then $\varphi(Cb)=(0,e_v^\vee)\varphi(b)$.  If $b$ is
an arrow in $Q$, then $(0,e_v^\vee)(b,0)=(0,be_v^\vee)$.  If $(0,be_v^\vee) \ne 0$ then
$(0,be_v^\vee)=(0,r^\vee)$ where $rb=e_v$, which is not possible since $b$ is an arrow.
If $b=a_m$, for some $m\in\cM$, then    $\varphi(Cb)=(0,e_v^\vee)(0, m^\vee)=
0$ by Lemma \ref{basic properties of T(A)}.  Hence we have shown that
Type 2 relations are sent to $0$ under $\varphi$.

Finally, let $ab$ be a Type 3 relation.  We want to show that $\varphi(ab)=0$.
There are 4 cases:  both $a$ and $b$ are arrows in $Q$,  $a$ is an arrow in $Q$
and $b=a_m$ for some $m\in \cM$,  $b$ is an arrow in $Q$
and $a=a_m$ for some $m\in \cM$, and $a=a_m$, $b=a_{m'}$ for some $m,m'\in \cM$.
If  both $a$ and $b$ are arrows in $Q$,  then since $ab$ is a relation in $A^*$ and
hence $ab=0$.  Next suppose that  $a$ is an arrow in $Q$
and $b=a_m$ for some $m\in \cM$.  Then $\varphi(aa_m)=(a,0)(0,m^\vee)=(0,am^\vee)$.
If $(0,am^\vee)\ne 0$, then $(0,am^\vee)=(0,r^\vee)$ where  $ra=m$.  But then
$a$ is the last arrow in $m$ and $aa_m$ is not a Type 2 relation.  The
case where $a=a_m$ for some $m\in \cM$ and $b$ is an arrow is handled
in a similar fashion to the last case.  The final case is when $a=a_m$ and
$b=a_{m'}$ for some $m,m'\in\cM$.  Then $\varphi(ab)=(0,m^\vee)(0,{m'}^\vee)=0$
by \ref{basic properties of T(A)}(4).  This completes the proof that $\varphi(I^*)=0$.

Since $\varphi\colon KQ^*\to T(A)$ is a surjection and $\varphi(I^*)=0$, $\varphi$
\sloppy induces a surjection $\psi\colon KQ^*/I^*\to T(A)$.   Now $A^*=KQ^*/I^*$ 
and by Proposition \ref{dim of A*},  $\dim_K(A^*)=2\dim_K(A)$.  Clearly,
$\dim_K(T(A)=2\dim_K(A)$.   Hence $\psi\colon A^*\to T(A)$ is an isomorphism. 
\end{proof}

\begin{corollary} Let $A=KQ/I$ be an almost gentle algebra. Then $T(A)$ is symmetric special multiserial, that is $T(A)$ is a Brauer configuration algebra. 
\end{corollary}

We end this section with an open question.

\begin{Question}{\rm 
Is it true that if a trivial extension $T(A)$ of a finite dimensional $K$-algebra is special multiserial then $A$ is almost gentle? }
\end{Question}

\section{Admissible cuts}\label{sec:admissible cuts}
Let $\Lambda=KQ_{\Lambda}/I_{\Lambda}$ be the  symmetric special multiserial algebra given by the defining
pair $(\cS,\mu)$.  There is an equivalence
relation on $\cS$ given by two special cycles are equivalent if one is a cyclic permutation of the
other.  Let $\{C_1,\dots,C_t \}$ be a set of equivalence class representatives.

\begin{definition}
{\rm An \emph{admissible cut $D$} of $Q_{\Lambda}$ is a subset of arrows in $Q_{\Lambda}$ consisting of exactly one arrow  in each special cycle corresponding to   an equivalence class representative  $C_i$, for $i = 1, \ldots, t$. We call 
$kQ_{\Lambda}/ \langle I_{\Lambda}\cup D\rangle$ the \emph{cut algebra associated to $D$} where $\langle I_{\Lambda} \cup D\rangle$ is the ideal generated by $I_{\Lambda} \cup D$.  }
\end{definition}

 Recall that any symmetric special multiserial algebras is defined by cycles. We show the following theorem.

\begin{theorem}\label{CutAlmostGentle}
Let $\Lambda = KQ_{\Lambda}/I_{\Lambda}$ be a  symmetric special multiserial algebra defined by a defining
pair $(\cS,\mu)$ and let $D$ be an admissible cut of $Q_{\Lambda}$. Set  
$Q$ to be the quiver given by $Q_0 = (Q_{\Lambda})_0$ and $Q_1 = (Q_{\Lambda})_1 \setminus D$ . 
Then the cut algebra, $KQ_{\Lambda} / \langle I_{\Lambda} \cup D\rangle$, associated to $D$
is isomorphic to $KQ /( I_{\Lambda}\cap KQ)$. 

Moreover, $KQ/( I_{\Lambda}\cap KQ)$ is an almost gentle algebra. 
\end{theorem}

\begin{proof}
The inclusion of quivers $Q \subset Q_{\Lambda}$ induces a $K$-algebra \sloppy  homomorphism $f : KQ \rightarrow  KQ_{\Lambda} / 
\langle I_{\Lambda} \cup D\rangle$. We show that $f$ is surjective. Let $\sum_p \lambda_p p $ be an element in $KQ_{\Lambda}$; that is, $\lambda_p\in K$, with almost all $\lambda_p=0$ and $p$ a path in $Q_{\Lambda}$. Then 
$\sum_p \lambda_p p = \sum_{q} \lambda_q q + \sum_r \lambda_r r $ where the first sum runs over all paths $q$ such that no arrow of $D$ occurs in $q$ and the second sum runs over all paths $r$ in $Q$ such that there is at least one arrow of $D$ in $r$. Then $\sum \lambda_r r$ is in the ideal $
\langle I_{\Lambda} \cup D\rangle$ and 
$\sum \lambda_q q $ is in the image of $KQ$ in $KQ_{\Lambda}$. It follows that $f$ is surjective. 

We now show that $\ker f = I_{\Lambda} \cap KQ$. 
Clearly $I_{\Lambda} \cap KQ \subset \ker f$. Now suppose that $f(\sum \lambda_p p) = 0$. Then 
$\sum \lambda_p p $ is in $  \langle I_{\Lambda} \cup D\rangle$.
Thus \[\sum\lambda_pp=\]
\[\sum \lambda_{r,s}r(C^{\mu(C)}-(C'^{\mu(C')}))s + \sum 
\lambda_{r',s'} r'C^{\mu(C)}as'
+\sum \lambda_{r'',s''}r''abs'' + \sum\lambda_{r''',s'''} r'''a_ds''',\]
where the $\lambda_{*,*}$ are elements of $K$, the $r,r',r'',r''',s,s',s'',s'''$ are paths,
the $C^{\mu(C)}-(C'^{\mu(C')})$ are Type 1 relations, the $C^{\mu(C)}a$ are Type 2
relations, the $ab$ are Type 3 relations, and the $a_d$ are arrows in  $ D$.
Since the left hand side is a $K$-linear combination of paths in $Q$, the sum of all
paths having at least one arrow in $D$ on the right hand side must equal $0$.  Each $C\in\cS$
has an arrow in $D$, so we conclude that
\[\sum\lambda_pp=\sum \lambda_{r'',s''}r''abs'',\]
where $ab$ is a Type 3 relation and no arrow in $D$ occurs in any $r''abs''$.
Noting that such $ab$ are in  $I_{\Lambda}\cap KQ$, we conclude that $f$ is an isomorphism.

It also follows from the above that the relations in $I_{\Lambda}\cap KQ$ are monomial quadratic. Suppose $ab \notin I_{\Lambda}\cap KQ$ and $a b' \notin I_{\Lambda}\cap KQ$ for $a, b, b' \in Q_1$. Then 
  $ab \notin I_{\Lambda}$ and $ab' \notin I_{\Lambda}$ which is a contradication since by \cite{GS2} 
 $KQ_{\Lambda}/ I_{\Lambda}$ is special multiserial.  Similarly we see that, given an arrow
in $Q$, there is at most one arrow $c\in Q_1$ such that $ca\notin I_{\Lambda}\cap KQ$.  Hence, $KQ/(I_{\Lambda} \cap
KQ)$ is
a special multiserial algebra and we have shown that $KQ/(I_{\Lambda}\cap KQ)$ is an almost gentle algebra. 
\end{proof}

The next result shows that if one starts with an almost gentle algebra and takes the appropriate admissible
cut in the trivial extension of the almost gentle algebra then the almost gentle algebra is isomorphic to
the cut algebra. 

\begin{theorem} 
Let $A = KQ/I$ be an almost gentle algebra with set of maximal paths $\cM$ and let $T(A)=Q_{T(A)}$ be the trivial extension of $A$ by $D(A)$ where the set of new arrows
of $Q_{T(A)}$ is given by $D = \{ \beta_m, m \in \cM \}$. Then $D$ is an admissible cut of $Q_{T(A)}$ and the cut algebra associated to $D$ is 
isomorphic to $A$. 
\end{theorem}

\begin{proof}
It follows from the construction  of $T(A)$ that there exists exactly one arrow from $D$ in any special cycle. Hence $D$ is an admissible 
cut of $T(A)$. The constructions in Theorem~\ref{CutAlmostGentle} give the result. 
\end{proof}

The next result shows that if one starts with a symmetric special multiserial algebra
defined by a defining pair $(\cS,\mu)$ and $\mu\equiv 1$ and an admissible cut $D$
then the algebra associated to $D$, trivially extended by its dual, is isomorphic to the original
symmetric special multiserial algebra.

\begin{theorem}\label{ext:trivial extension}
Let $\Lambda =KQ_{\Lambda} / I_{\Lambda}$ be a symmetric special multiserial algebra defined
by the defining pair
$(\cS,\mu)$ and assume that  $\mu\equiv 1$.
Let $D$ be an admissible cut of $Q_{\Lambda}$. Denote by $A =KQ/I$
the cut algebra associated to $D$. Then $T(A)$ is isomorphic to $\Lambda$. 
\end{theorem}

\begin{proof}
 The special cycles in $Q_{\Lambda}$ are of the form $C = p_1  \beta p_2$, for $\beta \in D$ and where $p_1 = a_1 \ldots a_r$ and 
$p_2 =   a_{r+1} \ldots a_s$. We now show that $p_2 p_1$ is a maximal path in $A$. Since $C$ is  a special cycle, we have special cycles $p_2p_1\beta$ and $\beta p_2p_1$.  Thus, $p_2p_1\notin I_\Lambda$ and
hence $p_2p_1\notin I_{\Lambda}\cap KQ$.  Since $\Lambda$ is a special multiserial algebra,
and both $a_r\beta$ and $\beta a_{r+1}$ are not in $I_{\Lambda}$, we see that
$a_rb$ and $ba_{r+1}$ are in $I_{\Lambda}$ for all arrows $b\in Q$.  Thus
$A$ is an almost gentle algebra since $I=I_{\Lambda}\cap KQ$ is generated by quadratic monomials
and is special multiserial.  It is now easy to see that $T(A)$ is isomorphic to $\Lambda$.

\end{proof}

Consider the set of pairs $(\Lambda, D)$ such that $\Lambda=KQ_{\Lambda}/I_{\Lambda}$ is a symmetric special
multiserial $K$-algebra and $D$ is an admissible cut in $Q_{\Lambda}$.   We say $(\Lambda,D)$ and
$(\Lambda',D')$ are \emph{equivalent} if there is a $K$-algebra isomorphism from $\Lambda$ to $\Lambda'$ sending
$D$ to $D'$ and let $\cX$ denote the equivalent classes. The next result is an immediate consequence
of the previous two theorems.

\begin{corollary}
There is a bijection $\varphi: \cA \longrightarrow \cX$ from the set $\cA$ of isomorphism classes of almost gentle algebras to the set $\cX$ of equivalence classes of pairs consisting of a symmetric special multiserial algebra and a cut as defined above. The isomorphism is given, for $A \in \cA$, by $\varphi (A) = (T(A), D)$ where $D = \{\beta_m\mid m\text{ a maximal path in }A\}$. Moreover, for $(\Lambda,D) \in \cX$, we have $\varphi^{-1}(\Lambda,D)  = A$ where $A$ is the isomorphism class of the algebras associated to the cut $D$.


\end{corollary}

\begin{remark}
{\rm 1) Given a symmetric special multiserial algebra $\Lambda = KQ_{\Lambda}/ I_{\Lambda}$, two distinct admissible cuts of $Q_{\Lambda}$ yield, in general, 
non-isomorphic, non derived equivalent cut algebras $A$ and $A'$. We note that $A$ and $A'$ have the same number of simple 
modules and $\dim_K A = \dim_K A'$. But there are examples where ${\rm gldim} A < \infty $ and ${\rm gldim} A' = \infty $.

2) If $\Lambda=KQ_{\Lambda}/I_{\Lambda}$ is of finite (resp. tame) representation type then any cut algebra associated to a cut of $Q_{\Lambda}$ is of finite (resp. tame) 
representation type.  To see this, suppose that $A$ is the cut algebra of $\Lambda$ associated to an admissible
cut.  Then $\Lambda$ is isomorphic to $T(A)$ and there is a full
faithful embedding of the category of finitely generated $A$-modules into the category
of finitely generated $\Lambda$-modules.
}
\end{remark}

Let $\Lambda=KQ_{\Lambda}/I_{\Lambda}$ be a symmetric special multiserial algebra and $(\cS,\mu)$ be a defining pair 	in $Q_{\Lambda}$
so that $\Lambda$ is defined by $(\cS,\mu)$.  If $\mu$ is identically equal to 1, we say that
\emph{$\Lambda$ has multiplicity function identically equal to 1}. Note that if one views $\Lambda$ as a Brauer configuration algebra with multiplicity function $\nu$, this corresponds to  $\nu$ being identically equal to one. 

\begin{corollary} 
Every symmetric special multiserial algebra with multiplicity function identically equal to one in its defining pair  is a trivial extension of an almost gentle algebra. 
\newline
 Equivalently, we have that every Brauer configuration algebra with multiplicity function identically equal to one is the trivial extension of an almost gentle algebra. 
\end{corollary}

\section{The Brauer configuration algebra associated to an almost gentle algebra.}\label{sec-bca}
We have seen that to every almost gentle algebra $A=K\cQ/I$, the trivial extension algebra $T(A)$  
is a symmetric special multiserial algebra.   In \cite{GS2} we saw that a symmetric
special multiserial algebra is a Brauer configuration algebra.  In this section we show how to
construct the Brauer configuration of the Brauer configuration algebra $T(A)$ from an
almost gentle algebra $A$.

Given an almost gentle algebra $A$ in Section \ref{symm spec
assoc to A}  we saw that there is
a  defining pair $(\cS,\mu)$ obtained
from $A$ and the algebra associated to $(\cS,\mu)$ is
isomorphic to $T(A)$.  
In this section, we show, more generally, how to construct a Brauer configuration from a defining
pair $(\cT,\mu)$ so that the associated Brauer configuration algebra is isomorphic to
the algebra associated to $(\cT,\mu)$.

Let $(\cT,\mu)$ be a defining pair for the quiver $\cQ$.  There is an equivalence relation on $\cT$ given
by two cycles in $\cT$ are equivalent if one is a cyclic permutation of the other.
Suppose there are $m$ equivalence classes of elements of $\cT$ 
and let $c_1,\dots,c_m$ be a full set representatives of the equivalence classes.

Recall from \cite{GS1} that a \emph{Brauer configuration} is a 4-tuple, $\Gamma=(\Gamma_0,
\Gamma_1,\nu,\mathfrak o)$, where $\Gamma_0$ is a set of vertices, $\Gamma_1$ is a
set of polygons which are multisets of vertices, $\nu\colon \Gamma_0\to\mathbb Z_{\ge 0}$, and
$\mathfrak o$ is an orientation. 
We begin by setting $\Gamma_0=\{\alpha_1,\cdots, \alpha_m\}$ where $m$
is the number of equivalence classes of elements  of  $\cT$.  If $\cQ_0=\{v_1,\dots,v_n\}$ then
$\Gamma_1=\{V_1,\dots, V_n\}$  where $\alpha_j$ occurs $k$-times in the multiset $V_i$ if
$v_i$ occurs as a vertex $k$-times in the cycle $c_j$.  The function  $\nu$ is defined
by $\nu(\alpha_i)=\mu(c_i)$.  Finally, the orientation at vertex $\alpha_i$ is
$V_{i_1}<V_{i_2}<
\cdots <V_{i_{\ell(c_i)}}(<V_{i_1})$ if the sequence of vertices in the cycle $c_i$ is
$v_{i_1},v_{i_2},\dots, v_{\ell(c_i)},v_{i_1}$.

It is straightforward to check that the Brauer configuration algebra associated to the
Brauer graph $(\Gamma_0,\Gamma_1,\nu, \mathfrak o)$ defined above is isomorphic to the algebra associated to the defining pair $(\cT,\mu)$.

\section{The oriented hypergraph of an almost gentle algebra}

This section builds on the observation that every Brauer configuration $\Gamma$ is an oriented hypergraph $\mathcal H$ with a vertex decoration where the decoration on $\cH$ corresponds to the multiplicity function on $\Gamma$ and where the orientation on $\mathcal H$ is induced by the orientation on $\Gamma$.  Given an almost gentle algebra, we will give an alternative direct construction of its associated oriented hypergraph (i.e. without passing to the trivial extension). 
This construction gives, for example, an easy criterium to determine whether two almost gentle algebras have isomorphic trivial extensions. 

A {\it hypergraph} is a generalisation of a graph in which an edge can contain more than two vertices. That is a hypergraph $\cH$ is a tuple $(\cH_0, \cH_1)$ where $\cH_0$ is a finite set of vertices and $\cH_1$ is a finite set of hyperedges given by multisets of elements of $\cH_0$ with the convention that each multiset contains at least two elements. A {\it  hypergraph with orientation} is a hypergraph $\cH = (\cH_0, \cH_1)$ together with an orientation $\sigma$ such that for every vertex $x \in \cH_0$, the set of hyperedges containing $x$ are cyclically ordered (counting
repeats). 

\begin{remark} {\rm In the context of Brauer configurations, we also could adopt the convention of allowing hyperedges with one element. These would correspond to the truncated edges of the Brauer configuration. }
\end{remark}

Let $A = KQ/I$ be an almost gentle algebra. Recall from Section 2 that  $\cM$ is the set of maximal paths in $KQ/I$.

Let $V_0$ be the subset of $Q_0$ containing vertices $v$ such that one of the following holds: 
\begin{enumerate}
\item  $v$ is the source of exactly one arrow and there is no arrow ending at $v$,
\item $v$ is the target of exactly one arrow and there is no arrow starting at $v$, 
\item $v$ is the target of exactly one arrow $a$ and the source of exactly one arrow $b$ and $ab \notin I$. 
\end{enumerate}
Set $\overline{\cM} = \cM \cup \{e_v \vert v \in V_0\}$. We say that a vertex $v \in Q_0$ {\it  lies in} $\cM$, if there exists $p \in \overline{\cM}$ with $p = q e_v r$ where $q,r$ are possibly trivial paths in $Q$. 

The following follows directly from the definition of $\overline{\cM}$. It follows from Lemma~\ref{unique maximal path} that

\begin{lemma} 
Every vertex in $Q_0$ lies in at least two elements of $\overline{\cM}$. 
 \end{lemma}

\emph{ {\bf  Construction of the  hypergraph $\cH_A$ with orientation of $A$:}} Let $A =  KQ/I$ be an almost gentle algebra.  Define  $\cH_A = (\cH_0, \cH_1)$ as follows. 
\begin{itemize} 
\item The vertices $\cH_0$ are in correspondence with the elements in $\overline{M}$. 
\item The hyperedges in $\cH_1$  correspond to the vertices in $Q_0$; namely, the hyperedge corresponding to a vertex $v \in Q_0$ is given by all elements $p \in \overline{M}$ such that $v$ lies in $p$. 
\item The orientation is induced by the maximal paths in $M$: Let $x$ be a vertex in $\cH_0$ and let $V_1, V_2, \ldots, V_n$ be the hyperedges corresponding respectively to the vertices $v_1, v_2, \ldots, v_n$  in $Q_0$ such that $v_1, v_2, \ldots, v_n$ lie in (the maximal path) $p$ corresponding to $x$. Suppose, without loss of generality, that $p= e_{v_1} a_1 e_{v_2}  a_2 e_{v_3} \cdots a_{n-1} e_{v_{n}}$ with $a_i \in Q_1$ then the cyclic ordering at $x$ is given by $V_1 < V_2 < \cdots < V_n < V_1$. 

\end{itemize}

Note that if $\Gamma$ is the Brauer configuration corresponding to $\cH$ then the multiplicity function of $\Gamma$ is identically equal to one and by our results 
the associated Brauer configuration algebra $\Lambda_\Gamma$ is isomorphic to the the trivial extension $T(A)$.

\begin{remark}{\rm 
1) If $A$ is gentle then the construction of the oriented hypergraph gives exactly the ribbon graph constructed in \cite{S}. We note that this is the general construction underlying the graphs in \cite{Parsons-Simoes, Garver}.

2) In the case of a gentle algebra associated to a surface triangulation (resp. angulation of a surface), the ribbon graph associated to the gentle algebra $A$ gives rise to the underlying surface and its triangulation (resp. angulation).
}\end{remark}

The hypergraph of an almost gentle algebra $A$, uniquely determines a Brauer configuration which uniquely determines a Brauer configuration algebra. It follows from Sections~\ref{sec:trivial extension}  and~\ref{sec-bca} that  this Brauer configuration algebra is isomorphic to the trivial extension of $A$.  Hence, we see immediately that

\begin{theorem}
Two almost gentle algebras $A$ and $B$ have the same associated hypergraph with orientation if and only if $T(A) \simeq T(B)$. 
\end{theorem}

\begin{Examples}{\rm 
1) Let $A_1=KQ_1 / I_1$ be the almost gentle algebra given by

\begin{tikzcd} Q_1: & 
   v_1 \ar[%
    ,loop 
    ,out=180 
    ,in=120 
    ,distance=2.5em 
   , "a_1" ] \arrow[r,bend left,"a_2"] &  \arrow[l, bend left, shift left=1.5ex, "b"] \arrow[l,bend left,"a_3"'] v_2  \arrow[r,bend left,"c"] & v_3  \\
\end{tikzcd}
and $I = \langle a_2 b, a_2 c, a_3 a_1, b a_1, b a_2 \rangle$. Then 
$$\overline{M} = \{ a_1 a_2 a_3, b, c\} \cup  \{e_{v_3}\}$$ 
Therefore $\cH_A = (\cH_0, \cH_1)$ is such that $\cH_0 = \{1,2,3,4 \}$ where 
\begin{flalign*}
1 &\mbox{ corresponds to } a_1 a_2 a_3 \\
2 &\mbox{ corresponds to } b \\
3 &\mbox{ corresponds to } c \\
4 &\mbox{ corresponds to } e_{v_3} 
\end{flalign*} 
and $\cH_1 = \{ V_1, V_2, V_3 \}$ where
\begin{align*}
V_1 & = \{ 1, 1, 2\} \\
V_2 & = \{ 1, 2, 3\} \\
V_3 & = \{ 3, 4\} \\
\end{align*}
Finally the orientation is induced by the order of the vertices in the maximal paths:  $e_{v_1} a_1 e_{v_1} a_2 e_{v_2} a_3 e_{v_1}$, $e_{v_2} b e_{v_1}, e_{v_2} c e_{v_3}$.  So the cyclic ordering of the polygons at vertex 1 is given by $V_1 < V_1 < V_2 < V_1 < V_1$, at vertex 2 it is $V_2 < V_1 <V_2$, at vertex 3 it is $V_2 < V_3 <V_2$ and at vertex 4 it is $V_3$.   

Let 
\begin{tikzcd} Q_2: & 
   v_1 \ar[%
    ,loop 
    ,out=180 
    ,in=120 
    ,distance=2.5em 
   , "a_2" ]  
   \ar[%
    ,loop 
    ,out=250 
    ,in=190 
    ,distance=2.5em 
   , "a_1" ]  
   \arrow[r,bend left,"a_3"] \arrow[r,bend left,  shift right=1.5ex, "b"'] &   v_2   & \arrow[l,bend left,"c"] v_3  \\
\end{tikzcd} and $I_2 = \langle a_1 a_3, a_1 b, a_2 a_1, a_2 b \rangle$. 
Note that the algebra $A_2 = KQ_2 /I_2$ has the same associated hypergraph as $A_1$, that is $\cH_{A_2} = \cH_{A_1}$. Therefore by Theorem 4.3 and the construction in Section 6, the algebras $T(A_1)$ and $T(A_3)$ are isomorphic. }\end{Examples}

\bibliographystyle{plain}

\end{document}